\documentclass[twoside]{amsart}
\usepackage{amsfonts,amsthm}
\usepackage{enumerate}
\usepackage{amsmath}
\usepackage{amssymb}
\usepackage{amsxtra}
\usepackage{latexsym}

\theoremstyle{plain}
\newtheorem*{thm A}{Theorem~A}
\newtheorem*{thm B}{Theorem~B}
\newtheorem*{thm C}{Theorem~C}
\newtheorem*{thm D}{Theorem~D}
\newtheorem*{coro1}{Corollary~1}
\newtheorem*{coro2}{Corollary~2}

\newtheorem*{main1}{Theorem~1}
\newtheorem*{main2}{Theorem~2}

\newtheorem*{pro A}{Proposition~A}
\newtheorem*{pro B}{Proposition~B}
\newtheorem*{lem A}{Lemma~A}
\newtheorem*{lem B}{Lemma~B}
\newtheorem*{lem C}{Lemma~C}
\newtheorem*{lem D}{Lemma~D}
\newtheorem*{lem 1}{Lemma~1}
\newtheorem*{lem 2}{Lemma~2}
\newtheorem*{lem 3}{Lemma~3}
\newtheorem*{rem 1}{Remark~1}
\newtheorem*{rem 2}{Remark~2}
\newtheorem*{rem 3}{Remark~3}

\newtheorem*{proof*}{\it The proof of Corollary}
\newtheorem{theorem}{Theorem}[section]
\newtheorem{lemma}[theorem]{Lemma}

\newtheorem{remark}[theorem]{Remark}

\newtheorem*{defn A}{Definition~A}
\newtheorem*{defn B}{Definition~B}

\def \QP{{\mathcal Q}^{\bot}}
\def \Q{\mathcal Q}
\def \N{\nabla}
\def \EN{{\eta}_{\nu}}

\def \XN{\xi_{\nu}}
\def \EoK{{\eta}_1({\xi})}
\def \ENK{{\eta}_{\nu}({\xi})}
\def \SN{\sum_{\nu=1}^3}

\def \al{\alpha}
\def \be{\beta}

\def \e{\eta}
\def \E{\eta}
\def \la{\lambda}
\def \eo{\eta_{1}}

\def \eh{\eta_{3}}
\def \etw{\eta_{2}}
\def \eth{\eta_{3}}

\def \x{\xi}
\def \xo{{\xi}_1}
\def \xt{\xi_{2}}
\def \xh{\xi_{3}}
\def \xtw{{\xi}_2}
\def \xth{{\xi}_3}
\def \XN{\xi_{\nu}}

\def \p{\phi}
\def \po{\phi_{1}}
\def \pn{\phi_{\nu}}
\def \PN{\phi_{\nu}}

\def \Ph{\phi}

\def \pth{\phi_{3}}

\def \GBt{G_2({\mathbb C}^{m+2})}
\def \GBo{G_2({\mathbb C}^{m+1})}

\def \gtw{\widehat \nabla ^{(k)}}

\def \la{\lambda}

\begin{document}

\title[Recurrent Ricci tensors]{Real Hypersurfaces in complex two-plane Grassmannians with recurrent Ricci tensor}
\vspace{0.2in}

\author[Y.J. Suh, D.H. Hwang, and C. Woo]{Young Jin Suh, Doo Hyun Hwang, and Changhwa Woo}
\address{\newline
Y.J. Suh, D.H. Hwang, and C. Woo
\newline Department of Mathematics,
\newline Kyungpook National University,
\newline Daegu 702-701, REPUBLIC OF KOREA}
\email{yjsuh@knu.ac.kr}
\email{engus0322@knu.ac.kr}
\email{legalgwch@knu.ac.kr}

\footnotetext[2]{{\it Key words and phases}: Hopf hypersurfaces,
complex two-plane Grassmannians, Recurrent Ricci tensor}

\thanks{*The first author was supported by by Grant Proj. No. NRF-2012-R1A2A2A01043023 and the third author by NRF-2013-Fostering Core Leaders of the Future Basic Science Program funded by the Korean Government.}

\begin{abstract}
In this paper, we have introduced a new notion of generalized Tanaka-Webster Reeb recurrent Ricci tensor in complex two-plane Grassmannians $G_2({\mathbb C}^{m+2})$.
Next, we give a non-existence property for real hypersurfaces $M$ in $G_2({\mathbb C}^{m+2})$ with such a condition.
\end{abstract}

\maketitle
\section*{Introduction}
\setcounter{equation}{0}
\renewcommand{\theequation}{0.\arabic{equation}}
\vspace{0.13in}

The complex two-plane Grassmannians $\GBt$ is a kind of Hermitian symmetry spaces of compact irreducible type with rank~$2$. It consists of all complex two-dimensional linear subspaces in ${\mathbb C}^{m+2}$. Remarkably, it is equipped with both a K\"{a}hler structure $J$ and a quaternionic K\"{a}hler structure ${\mathfrak J}$ (not containing $J$) satisfying $JJ_{\nu}=J_{\nu}J$ $(\nu=1,2,3)$, where $\{J_{\nu}\}_{\nu=1,2,3}$ is an orthonormal basis of $\mathfrak J$. When $m=1$, $G_2({\mathbb C}^3)$ is isometric to the two-dimensional
complex projective space ${\mathbb C}P^2$ with constant holomorphic
sectional curvature eight. When $m=2$, we note that the isomorphism $\text{Spin}(6) \simeq \text{SU}(4)$ yields an isometry between $G_2({\mathbb C}^4)$ and the real Grassmann
manifold $G_2^+({\mathbb R}^6)$ of oriented two-dimensional linear
subspaces in ${\mathbb R}^6$. In this paper, we assume $m \geq 3$ (see Berndt and Suh~\cite{BS1} and~\cite{BS2}).
\vskip 3pt
Let $M$ be a real hypersurface of $G_2({\mathbb C}^{m+2})$, that is, a submanifold of $G_2({\mathbb C}^{m+2})$ with real codimension one and $T_pM$ stands for the tangent space of $M$ at $p\in M$. The
induced Riemannian metric on $M$ will also be denoted by $g$, and $\nabla$ denotes the Riemannian connection of $(M,g)$. Let $N$ be a
local unit normal vector field of $M$ and $A$ the shape operator of $M$ with respect to $N$. By the formulas in \cite[Section~$2$]{LCW}, it can be easily seen that $\x$ is Hopf if and only if $M$ is Hopf. From the quaternionic K\"{a}hler structure $\mathfrak J$ of $\GBt$, there naturally exist {\it almost contact 3-structure} vector fields defined by $\x_{\nu}=-J_{\nu}N$, $\nu=1,2,3$. Next, let us denote by $\QP=\text{Span}\{\,\xo, \xt, \xh\}$ a 3-dimensional distribution in a tangent space $T_{p}M$ at $p \in M$, where $\Q$ stands for the orthogonal complement of $\QP$ in $T_{p}M$.
Thus the tangent space of $M$ at $p\in M$ consists of the direct sum of $\Q$ and $\QP$, that is, $T_{p}M =\Q\oplus \QP$. (see \cite{BS1} and \cite{LSW}).
\vskip 5pt
By using the result of Alekseevskii~\cite{Al-01}, Berndt and Suh~\cite{BS1} have classified all real hypersurfaces with these invariant properties in $\GBt$ as follows:
\vskip 5pt
\begin{thm A}\label{thm1}
Let $M$ be a real hypersurface in $\GBt$, $m \geq 3$. Then both $[\x]$ and $\QP$ are invariant under the shape operator of $M$ if and only if
\begin{enumerate}[\rm(A)]
\item {$M$ is an open part of a tube around a totally geodesic $\GBo$ in $\GBt$, or}\
\item {$m$ is even, say $m = 2n$, and $M$ is an open part of a tube around a totally geodesic ${\mathbb H}P^n$ in $\GBt$}.
\end{enumerate}
\end{thm A}

In the case of~$(A)$ (resp.,~$(B)$), we say that $M$ is of Type~$(A)$ (resp., Type~$(B)$).
Furthermore, the real hypersurface $M$ is said to be {\it Hopf} if $A[\x]\subset[\x]$, or equivalently,
the Reeb vector field $\x$ is principal with principal curvature $\al=g(A\x,\x)$. In this case,
the principal curvature $\al=g(A\x,\x)$ is said to be a {\it Reeb curvature} of~$M$.
\vskip 5pt
By using Theorem~$\rm A$, many geometers have given some characterizations for Hopf hypersurfaces in $\GBt$ with geometrical quantities; shape operator, normal (or structure) Jacobi operator, Ricci tensor, and so on.  The Ricci tensor $S$ of $M$ in $\GBt$ is given by $$g(SX,Y)={\sum}_{i=1}^{4m-1}g(R(e_i,X)Y,e_i),$$ where $\{e_1, {\cdots}, e_{4m-1}\}$ denotes a basis of the tangent space $T_pM$ of $M$, $p{\in}M$, in $\GBt$ (see~\cite{S02}).
\vskip 5pt
Now we define the notion of recurrent, which is weaker than the usual parallelism.
The notion of recurrent for a $(1,1)$ type tensor field $T$ has a close relation to holonomy group. For a $1$-form $\omega$ on $M$ is defined by $\N T = T \otimes\omega$, (see \cite{KN}).
\vskip 5pt
Let us consider a notion of recurrent (resp., Reeb recurrent) Ricci tensor $S$ for a real hypersurface $M$ in $\GBt$ defined by
\begin{equation}\label{C-1}
\N_{X}S=\omega(X)S
\tag{C-1}
\end{equation}
for any $X$ in $TM$.
\vskip 5pt
Motivated by such a notion, we want to introduce another new notion of Reeb recurrent Ricci tensor. It is weaker than usual parallel Ricci tensor and is defined by
\begin{equation}\label{C-2}
\N_{\x}S=\omega(\x)S.
\tag{C-2}
\end{equation}
 Now we say that if $S$ satisfies the condition \eqref{C-2}, it is a {\it proper Reeb recurrent} if $\omega(\x)$ is non-vanishing, i.e., $\omega(\x)\neq 0$. Then \eqref{C-1} (resp., \eqref{C-2})
  means $[\N_XS,S]=\omega(X)[S,S]=0$ (resp., $[\N_{\x}S,S]=0 $)
for any tangent vector field $X$ defined on $M$ (see \cite{S96}). Its geometrical meaning is that the eigenspaces of the Ricci operator $S$ of M are
parallel along any curve $\gamma$ (resp., Reeb flow). Here, the eigenspaces are said to be parallel if they are invariant with respect to any parallel translations along $\gamma$ (resp., Reeb flow) (for detailed examples, see \cite{SM}, \cite{SM2}, \cite{KS}). There are many examples of Recurrent Ricci tensor in pseudo-Riemannian manifolds \cite[Example 4, p. 13]{SM2}.
\vskip 5pt
In this paper, we give a complete classification of real hypersurfaces $M$ in $\GBt$ with recurrent (resp., Reeb recurrent) Ricci tensor as follows:
\vskip 5pt
\begin{main1}\label{main1}
There do not exist any Hopf hypersurfaces in $\GBt$, $m \geq 3$, with proper Reeb recurrent Ricci tensor if the Reeb curvature is non-vanishing.
\end{main1}
\vskip 5pt
\begin{rem 1}
When $\omega(\x)=0$, the Reeb recurrent Ricci tensor is equivalent to Reeb parallel Ricci tensor, so by using the result of \cite{S01},
$M$ is locally congruent to one of the following:
\begin{enumerate}[\rm(i)]
\item {a tube over a totally geodesic $G_2({\mathbb C}^{m+1})$ in $G_2({\mathbb C}^{m+2})$ with radius $r \neq \frac{\pi}{4 \sqrt{2}}$, or} \
\item {a tube over a totally geodesic quaternionic projective
space ${\mathbb H}P^n$, $m = 2n$, in $G_2({\mathbb C}^{m+2})$ with radius $r$ such that $\cot^{2}(2r) = \frac{1}{2m-1}$ and $\x$-parallel eigenspaces $T_{\cot r}$ and $T_{\tan r}$}.
\end{enumerate}
\end{rem 1}
\vskip 5pt
On the other hand, if we use the result in \cite{S}, we can assert another non-existence property for real hypersurfaces in $\GBt$ with Recurrent Ricci tensor as follows:
\vskip 5pt
\begin{coro1}\label{coro1}
There do not exist any Hopf hypersurfaces in $\GBt$, $m \geq 3$, with recurrent Ricci tensor.
\end{coro1}
\vskip 5pt

Next, we consider a new connection which is different from the usual Levi-Civita connection, so called, the generalized Tanaka-Webster (in short, GTW) connection. Even though this connection
does not satisfies torsion free condition, it is deeply related to the contact structure (see \cite{CH1}, \cite{CH2}).\\
\par Let us consider a notion of the GTW recurrent Ricci tensor $S$ for a real hypersurface $M$ in $\GBt$ defined by
\begin{equation}\label{C-3}
\gtw_{X}S=\omega(X)S
\tag{C-3}
\end{equation}
for any $X$ in $TM$, where $\omega$ denotes a $1$-form defined on $M$.

Similarly, we may also define GTW Reeb parallel Ricci tensor as follows

\begin{equation}\label{C-4}
\gtw_{\x}S=\omega(\x)S.
\tag{C-4}
\end{equation}
We say that the condition \eqref{C-4} is said to be a {\it proper GTW Reeb recurrent} if the $1$-form $\omega(\x)$ is non-vanishing, i.e., $\omega(\x)\neq 0$.
We can classify real hypersurfaces $M$ in $\GBt$ with GTW Reeb recurrent Ricci tensor as follows:
\vskip 5pt
\begin{main2}\label{main2}
There do not exist any Hopf hypersurfaces in $\GBt$, $m \geq 3$, $(\al\neq 2k)$ with proper GTW Reeb recurrent Ricci tensor.
\end{main2}
\vskip 5pt

\vskip 5pt
\begin{rem 2}
When $\omega(\x)$ identically vanishes, that is, $\omega(\x)=0$, then the GTW Reeb recurrent Ricci tensor is equivalent to GTW Reeb parallel Ricci tensor; therefore, by using the result of \cite{LSW},
$M$ is locally congruent to one of the following:
\begin{enumerate}[\rm(i)]
\item {a tube over a totally geodesic $G_2({\mathbb C}^{m+1})$ in $G_2({\mathbb C}^{m+2})$ with radius $r$ such that  $r \neq \frac{1}{2\sqrt{2}}\cot^{-1}(\frac{k}{\sqrt{2}})$, or}
\item {a tube over a totally geodesic ${\mathbb H}P^n$, $m = 2n$, in $G_2({\mathbb C}^{m+2})$ with radius $r$ such that $r=\frac{1}{2}\cot^{-1}(\frac{-k}{4(2n-1)})$.}
\end{enumerate}
\end{rem 2}
\vskip 5pt
Using the result in \cite{PS2}, we can assert another non-existence property for real hypersurfaces $M$ in $\GBt$ as follows:
\vskip 5pt
\begin{coro2}\label{coro2}
There do not exist any Hopf hypersurfaces in $\GBt$, $m \geq 3$, $(\al\neq 2k)$ with GTW recurrent Ricci tensor.
\end{coro2}

In Sections~\ref{section 1}, \ref{section 2} complete proofs of Theorem~$1$ and Corollary~$1$ will be given respectively.
In Sections~\ref{section 3} and \ref{section 4}, the proofs of Theorem~$2$ and Corollary~$2$ will be given.
Main references for Riemannian geometric structures of $\GBt$, $m \geq 3$ will be explained in detail (see \cite{Al-01}, \cite{BS1}, \cite{BS2}, and \cite{LS}).
\vskip 5pt
\section{The proper Reeb recurrent Ricci tensor}\label{section 1}
\setcounter{equation}{0}
\renewcommand{\theequation}{1.\arabic{equation}}
\vspace{0.13in}

From now on, let $M$ represent a real hypersurface in $\GBt$, $m \geq 3$, and $S$ denote the Ricci tensor of $M$. Hereafter, unless otherwise stated, we consider that $X$ and $Y$ are any tangent vector fields on $M$ and $N$ denotes the normal vector field of $M$. $\omega$ stands for any $1$-form on $M$. For the K\"{a}hler structure $J$ and the quaternionic K\"{a}hler structure ${\mathfrak J}=\text{span}\{J_{\nu}\}_{\nu=1,2,3}$, we may put
\begin{equation*}
JX={\phi}X+{\eta}(X)N,\quad J_{\nu}X={\phi}_{\nu}X+{\eta}_{\nu}(X)N
\end{equation*}
where $\p X$ (resp., $\PN X$) is the tangential part of $JX$ (resp., $J_{\nu}X$) and $\e(X)=g(X,\x)$ (resp., $\EN(X)=g(X,\XN)$) is the coefficient of normal part of $JX$ (resp., $J_{\nu}X$). In this case, we call $\p$ the structure tensor field of $M$.
In \cite{PS}, the Ricci tensor $S$ of a real hypersurface $M$ in $\GBt$, $m \geq 3$, is given by
\begin{equation}\label{eq: 1.1}
\begin{split}
SX & = \sum_{i=1}^{4m-1} R(X,e_i)e_i \\
   & = (4m+7)X - 3\eta(X)\x + hAX - A^2 X \\
   & \quad + \sum_{\nu =1}^3 \Big\{ - 3 \eta_{\nu}(X)\x_{\nu} + \eta_{\nu}(\x)\p_{\nu}\p X-\eta_{\nu}(\p X)\p_{\nu}\x -\eta(X)\eta_{\nu}(\x)\x_{\nu} \Big\},
\end{split}
\end{equation}
where $h$ denotes the trace of the shape operator $A$, that is, $h=\text{Tr }A$.\\
In \cite{LSW}, the covariant derivative of $S$ is given by
\begin{equation}\label{eq: 1.2}
\begin{split}
(\N_{X}S)Y & = - 3g(\p AX, Y) \x - 3\eta(Y) \p AX \\
                 & \quad - 3\sum_{\nu=1}^{3} \big\{ g(\p_{\nu}AX, Y) \x_{\nu} + \eta_{\nu} (Y) \p_{\nu}AX \big\}\\
                 &\quad + \sum_{\nu=1}^{3}\Big\{2g(\p AX, \x_{\nu})\p_{\nu}\p Y + g(AX, \p_{\nu}\p Y)\p_{\nu}\x \\
                 &\quad \ \ - \eta(Y) g(AX, \x_{\nu})\p_{\nu}\x + \eta_{\nu}(\p Y) g(AX, \x) \x_{\nu}-\eta_{\nu}(\p Y) \p_{\nu} \p AX \\
                 &\quad \ \ - \eta(Y) g(\p AX, \x_{\nu})\x_{\nu}-\eta(Y) g(\p_{\nu}AX, \x)\x_{\nu}\Big\}\\
                 &\quad +(Xh)AY +h(\N_{X}A)Y - (\N_{X}A)AY -A(\N_{X}A)Y.
\end{split}
\end{equation}

Thus, $(\N_{X}S)Y=\omega(X)SY$ is embodied as follows:

\begin{equation}\label{eq: 1.3}
\begin{split}
                 &- 3g(\p AX, Y) \x - 3\eta(Y) \p AX \\
                 &- 3\sum_{\nu=1}^{3} \big\{ g(\p_{\nu}AX, Y) \x_{\nu} + \eta_{\nu} (Y) \p_{\nu}AX \big\}\\
                 & + \sum_{\nu=1}^{3}\Big\{2g(\p AX, \x_{\nu})\p_{\nu}\p Y + g(AX, \p_{\nu}\p Y)\p_{\nu}\x \\
                 &\quad- \eta(Y) g(AX, \x_{\nu})\p_{\nu}\x + \eta_{\nu}(\p Y) g(AX, \x) \x_{\nu}-\eta_{\nu}(\p Y) \p_{\nu} \p AX \\
                 &\quad- \eta(Y) g(\p AX, \x_{\nu})\x_{\nu}-\eta(Y) g(\p_{\nu}AX, \x)\x_{\nu}\Big\}\\
\end{split}
\end{equation}
\begin{equation*}
\begin{split}
                 &\quad +(Xh)AY +h(\N_{X}A)Y - (\N_{X}A)AY -A(\N_{X}A)Y\\
                 &=\omega(X)\Big[ (4m+7)Y - 3\eta(Y)\x + hAY - A^2 Y \\
                 & \quad + \sum_{\nu =1}^3 \big\{ - 3 \eta_{\nu}(Y)\x_{\nu} + \eta_{\nu}(\x)\p_{\nu}\p Y-\eta_{\nu}(\p Y)\p_{\nu}\x -\eta(Y)\eta_{\nu}(\x)\x_{\nu} \big\}\Big].
\end{split}
\end{equation*}

As a special case, we may consider Reeb directional derivative of the Ricci tensor. If the Ricci tensor of a real hypersurface $M$ in $\GBt$ is Reeb recurrent, then it is defined by
\begin{equation}\label{C-2}
(\N_{\x}S)Y=\omega({\x})SY.
\tag{C-2}
\end{equation}
Under the condition of being Hopf, \eqref{C-2} is specified:
\begin{equation}\label{eq: 1.4}
\begin{split}
                 &\quad- 3\al\sum_{\nu=1}^{3} \big\{  g(\p_{\nu}\xi, Y) \x_{\nu} +\eta_{\nu} (Y) \p_{\nu}\xi \big\}\\
                 &\quad+\al\sum_{\nu=1}^{3}\Big\{g(\xi, \p_{\nu}\p Y)\p_{\nu}\x-\eta(Y)\eta_{\nu}(\xi)\p_{\nu}\x +\eta_{\nu}(\p Y)\x_{\nu}\Big\}\\
                 &\quad+(Xh)AY +h(\N_{\xi}A)Y - (\N_{\xi}A)AY -A(\N_{\xi}A)Y\\
                 &=\omega(\x)\Big[ (4m+7)Y - 3\eta(Y)\x + hAY - A^2 Y \\
                 & \quad + \sum_{\nu =1}^3 \big\{ - 3 \eta_{\nu}(Y)\x_{\nu} + \eta_{\nu}(\x)\p_{\nu}\p Y-\eta(\p_{\nu}Y)\p_{\nu}\x -\eta(Y)\eta_{\nu}(\x)\x_{\nu} \big\}\Big].
\end{split}
\end{equation}

\par First of all, by using above assumption, we shall
show that the Reeb vector field $\x$ belongs to either the
distribution $\Q$ or the distribution $\Q^{\bot}$
such that $T_{x}M=\Q \oplus \QP$ for any
point $x \in M$.\\

\begin{lemma}\label{lemma 1.1}
Let $M$ be a Hopf hypersurface in $\GBt$, $m \geq 3$. If $M$ has Reeb recurrent Ricci tensor, then the Reeb vector field $\x$ belongs to either the distribution $\Q$ or the distribution $\QP$.
\end{lemma}

\begin{proof}
\noindent To show this fact, we consider that the Reeb vector field $\xi$ satisfies
\begin{equation}\label{xi}
\xi = \eta(X_{0})X_{0}+\eta(\xi_{1})\xi_{1}
\tag{*}
\end{equation}
 for some unit vectors $X_{0} \in
\Q$ and $\xi_{1} \in \QP$. Putting $Y=\x$ in \eqref{eq: 1.4}, by \eqref{xi} and using basic formulas in \cite[Section 2]{LCW}, it follows that
\begin{equation}\label{eq: 1.5}
\begin{split}
&-4\al \EoK\po\x +\al(\x h)\x + h(\x \al)\x + 2\al (\x \al)\x\\&
\quad\quad=\omega(\x)\big\{(4m+4+h\al+\al^2)\x-4\EoK\xo\big\},
\end{split}
\end{equation}
where we have used $(\N_{\x}A)\x = (\x \al) \x$ and $(\N_{\x}A)A\x=\al (\x \al) \x$.

\vskip 5pt

\noindent Taking the inner product of \eqref{eq: 1.5} with $\po\x$, we have
\begin{equation}\label{eq: 1.6}
-4\al\EoK\eta^{2}(X_0)=0.
\end{equation}
From this, we have the following three cases.

\vskip 3pt
\noindent {\bf Case 1\,:} \quad $\al = 0$.
\newline By the equation $Y{\alpha}=({\x}{\alpha}){\eta}(Y)-4{\SN}{\ENK}{\EN}({\p}Y)$
in \cite[Lemma~$1$]{BS1}, we obtain easily that $\x$ belongs to either $\Q$ or $\QP$ (see \cite{PS}).

\vskip 3pt
\noindent {\bf Case 2\,:} \quad $\eta(\x_{1})=0$.
\vskip 3pt

By the notation \eqref{xi} related to the Reeb vector field, we see that $\x$ belongs to the distribution $\Q$.

\vskip 3pt
\noindent {\bf Case 3\,:} \quad $\E(X_0)=0$.
\vskip 3pt

This case implies that $\x$ belongs to the distribution $\QP$.
\vskip 5pt
\noindent Accordingly, summing up these cases, the proof is completed.
\end{proof}

\begin{lemma}\label{lemma 1.2}
Let $M$ be a Hopf hypersurface in $\GBt$. If the Reeb vector field $\x$ belongs to $\QP$, then
the Ricci tensor $S$ and the shape operator $A$ commutes with each other, that is, $SA=AS$.
\end{lemma}

(see \cite[Lemma 1.2]{PaSW}.)
\vskip 3pt

\begin{lemma}\label{lemma 1.3}
\vskip 3pt
Let $M$ be a Hopf hypersurface in $\GBt$, $m \geq 3$, with non-vanishing Reeb curvature (i.e., $\al\neq0$).
When the Reeb vector field $\x$ belongs to $\QP$, if $M$ has Reeb recurrent Ricci tensor, that is, $(\N_{\x}S)X=\omega(\x)SX$, then $M$ must have commuting Ricci tensor $S\p=\p S$.
\end{lemma}

\vskip 3pt

\begin{proof}
From the Codazzi equation in \cite{BS1} and by differentiating $A\x=\alpha\x$, we obtain
\begin{equation}\label{eq: 1.13}
\begin{split}
(\N_\x A)X = (X\alpha)\x+\al \p AX-A\p AX +\p X+\po X+2\eth(X)\xtw-2\etw(X)\xth.
\end{split}
\end{equation}

\cite[Lemma~$A$, $(3.3)$]{LCW} is essential equation for proving this lemma:
\begin{equation}\label{eq: 1.14}
\begin{split}
\alpha   A\phi X + \alpha \phi A  X -2 A\phi AX  & + 2 \phi X =2\displaystyle \sum_{\nu=1}^3 \Big\{-\eta_\nu(X) \phi \xi_{\nu} - \eta_\nu(\phi X) \xi_{\nu}  \\
& - \eta_\nu(\xi) \phi_{\nu}X + 2\eta(X) \eta_\nu(\xi) \phi \xi_{\nu} + 2 \eta_\nu(\phi X)\eta_\nu(\xi) \xi \Big\}.
\end{split}
\end{equation}

\noindent Using \eqref{eq: 1.13} and \eqref{eq: 1.14}, we get $(\N_\x A)X=\frac{\al}{2}\p AX-\frac{\al}{2}A\p X+ (\x\al)\eta (X)\x$,
which changes \eqref{C-2} into
\begin{equation}\label{eq: 1.15}
\begin{split}
&(\x h)AX + \frac{h\al}{2}(\p AX- A\p X)-\frac{\al}{2}(\p A^2X- A^2\p X) +(h-2\al)(\x\al)\e(X)\x\\
&=\omega(\x)SX.
\end{split}
\end{equation}

Here replacing $X$ by $\p X$ in \eqref{eq: 1.1} \big(resp., applying $\p$ to \eqref{eq: 1.1}\big), we have

\begin{equation}\label{eq: 1.16}
\left \{
\begin{aligned}
S\p X=(4m+7)\p X-\po X+2\etw(X)\xth-2\eh(X)\xtw+hA\p X-A^2\p X, \\
\p SX=(4m+7)\p X-\po X+2\etw(X)\xth-2\eh(X)\xtw+h\p AX-\p A^2X.
\end{aligned}
\right.
\end{equation}

Combining equations in \eqref{eq: 1.16}, we obtain
\begin{equation}\label{eq: 1.17}
S\p X - \p SX = h A\p X - A^{2} \p X - h \p AX + \p A^{2}X.
\end{equation}
Using \eqref{eq: 1.17}, \eqref{eq: 1.15} becomes

\begin{equation}\label{eq: 1.18}
\begin{split}
&(\x h)AX +\frac{\al}{2}(\p SX-S\p X) +(h-2\al)(\x\al)\e(X)\x=\omega(\x)SX.
\end{split}
\end{equation}

Substituting $X$ to $A X$ into \eqref{eq: 1.18} and
applying $A$ to \eqref{eq: 1.18},
we have
\begin{equation}\label{eq: 1.19}
\left \{
\begin{aligned}
&(\x h)A^{2}X +\frac{\al}{2}(\p S-S\p)AX +\al(h-2\al)(\x\al)\e(X)\x=\omega(\x)SAX, \\
&(\x h)A^{2}X +\frac{\al}{2}A(\p S-S\p)X +\al(h-2\al)(\x\al)\e(X)\x=\omega(\x)ASX.
\end{aligned}
\right.
\end{equation}

By combining equations in \eqref{eq: 1.19} and using Lemma~\ref{lemma 1.2}, we get
\begin{equation}\label{eq: 1.21}
(\p S-S\p)A=A(\p S-S\p).
\end{equation}

%
If the Reeb vector field $\x$ belongs to $\QP$ and $A\x=\al\x$ on $M$, $A(\p S-S\p)=(\p S-S\p)A$ is equivalent to $S\p=\p S$ on $M$ (see \cite[Lemma 1.5]{PaSW}).

\end{proof}
\vskip 5pt
Summing up above lemmas~\ref{lemma 1.2},~\ref{lemma 1.3}, \cite[Theorem 1.1]{S02}, \cite[Theorem]{BS2} and \cite[Theorem]{BS1}, we conclude that if $M$ is a Hopf hypersurface in $\GBt$ on which holds \eqref{C-2}, then $M$ satisfies the condition of being a model space of Type~$(A)$ (shortly, $M_{A}$).
\vskip 5pt
From this together with Theorem~$A$ in the introduction we know that any
real hypersurface in $\GBt$ with Reeb recurrent Ricci tensor and
$\x\in \QP$ is congruent to a tube over a totally geodesic $\GBo$ in $\GBt$.
Now let us check if real hypersurfaces $M_{A}$ satisfy the condition of Reeb recurrent Ricci tensor.

By virtue of \cite{S01}, we have
\begin{remark}\label{remark 1.5}
\rm If $\omega(\x)=0$, the Ricci tensor $S$ of real hypersurfaces $M_{A}$ in $\GBt$ satisfies the {\it Reeb parallel} condition.
\end{remark}
So we may consider only $\omega(\x)\neq0$. We assume that $M_{A}$ satisfies \eqref{C-2}.\\
By the equation of Codazzi \cite{BS1} and \cite[Proposition~$\rm 3$]{BS1} we obtain $X \in T_{x}M_{A}=T_{\al}\oplus T_{\beta}\oplus T_{\lambda}\oplus T_{\mu}$
\begin{equation}\label{eq: 1.24}
(\N_{\x} S)X  = - h(\N_{\x}A)X + (\N_{\x}A)AX + A(\N_{\x}A)X
\end{equation}
and
\begin{equation}\label{eq: 1.25}
\begin{split}
(\N_{\x}A)X & =\frac{\al}{2}\p AX-\frac{\al}{2}A\p X+(\x\al)\e(X)\x\\
            & = \left\{ \begin{array}{ll}
                0                       & \mbox{if}\ \  X \in T_{\al},\\
                0                       & \mbox{if}\ \  X \in T_{\beta}=\text{span}\{\x_{\ell}|\, \ell=2,3\},\\
                0                       & \mbox{if}\ \  X \in T_{\lambda},\\
                0      & \mbox{if}\ \  X \in T_{\mu}.\\
\end{array}\right.
\end{split}
\end{equation}
From these two equations, it follows that
\begin{equation}\label{eq: 1.26}
(\N_{\x}S)X  = \left\{ \begin{array}{ll}
                0                    & \mbox{if}\ \  X=\xi \in T_{\al},\\
                0                    & \mbox{if}\ \  X=\x_{\ell} \in T_{\beta}=\text{span}\{\x_{\ell}|\, \ell=1,2,3\}, \\
                0                    & \mbox{if}\ \  X \in T_{\lambda},\\
                0                    & \mbox{if}\ \  X \in T_{\mu}.\\
\end{array}\right.
\end{equation}
\vskip 5pt
Consider
\begin{equation}\label{eq: 1.27}
SX  = \left\{ \begin{array}{ll}
                (4m+h\al-\al^2)\x                    & \mbox{if}\ \  X=\xi \in T_{\al},\\
                (4m+6+h\be-\be^2)\x_{\ell}                   & \mbox{if}\ \  X=\x_{\ell} \in T_{\beta}=\text{Span}\{\x_{\ell}|\, \ell=1,2,3\}, \\
                (4m+6+h\lambda-\lambda^2)X                   & \mbox{if}\ \  X \in T_{\lambda},\\
                (4m+8)X                    & \mbox{if}\ \  X \in T_{\mu}.\\
\end{array}\right.
\end{equation}
If we consider a non-zero tangent vector field $ X \in T_{\mu}$, then we get $\omega(\x)(4m+8)X=0$, which means $\omega(\x)=0$. This is a contradiction.
\begin{remark}\label{remark 1.6}
\rm If $\omega(\x)\neq0$, the Ricci tensor $S$ of real hypersurfaces $M_{A}$ in $\GBt$ does not satisfy the {\it Reeb recurrent} condition.
\end{remark}
\noindent Summing up all cases mentioned above, we can assert that if $\omega(\x)=0$, then $S$ of real hypersurfaces $M_{A}$ in $\GBt$ satisfies the Reeb recurrent condition.
\vskip 5pt
For $\x\in\Q$, by \cite[Main Theorem]{LS}, we know $g(A\Q,\QP)=0$.
\vskip 5pt
We know that a Hopf hypersurface $M$ in $\GBt$ with Reeb recurrent Ricci tensor and $\x\in\Q$ is a real hypersurface of type~$(B)$ (denoted by $M_{B}$) in $\GBt$, that is,
a tube over a totally geodesic ${\mathbb H}P^{n}$. We will check if such a tube satisfies the notion of Reeb recurrent Ricci tensor.
We assume that $M_{B}$ satisfies \eqref{C-2}.\\
\vskip 5pt
In order to do this, let us calculate the fundamental equation related to the covariant derivative of $S$ of $M_{B}$ along the direction of $\x$. On $T_{x}M_{B}$, $x \in M_{B}$, since $\x \in \Q$ and $h=\text{Tr}(A)=\al+(4n-1)\beta$ is a constant, equation \eqref{C-2} is reduced to
\begin{equation}\label{eq: 1.28}
\begin{split}
(\N_{\x}S)X & =  -4\al \sum_{\nu =1}^3 \Big \{\eta_{\nu}(\p X)\x_{\nu} -  \EN(X)\pn\x \Big\} \\
                & \ \ - h(\N_{\x}A)X + (\N_{\x}A)AX + A(\N_{\x}A)X.
\end{split}
\end{equation}
Moreover, by the equation of Codazzi \cite{BS1} and \cite[Proposition~$\rm 2$]{BS1}, we obtain that for any $X \in T_{x}M_{B}$ $=$
$T_{\al}\oplus T_{\beta}\oplus T_{\gamma}\oplus T_{\lambda}\oplus T_{\mu}$
\begin{equation}\label{eq: 1.29}
\begin{split}
(\N_{\x}A)X & =\al \p AX - A \p AX + \p X - \sum_{\nu=1}^{3}\big\{\e_{\nu}(X)\p_{\nu}\x + 3g(\p_{\nu}\x, X)\x_{\nu}\big\}\\
            & = \left\{ \begin{array}{ll}
                0                       & \mbox{if}\ \  X \in T_{\al}\\
                \al \beta \p \x_{\ell} & \mbox{if}\ \  X \in T_{\beta}=\text{Span}\{\x_{\ell}|\, \ell=1,2,3\}\\
                -4 \x_{\ell}           & \mbox{if}\ \  X \in T_{\gamma}=\text{Span}\{\p \x_{\ell}|\, \ell=1,2,3\}\\
                (\al \lambda + 2) \p X  & \mbox{if}\ \  X \in T_{\lambda}\\
                (\al \mu + 2) \p X      & \mbox{if}\ \  X \in T_{\mu}.\\
\end{array}\right.
\end{split}
\end{equation}
Combining \eqref{eq: 1.28} and \eqref{eq: 1.29}, it follows that
\begin{equation}\label{eq: 1.30}
(\N_{\x}S)X  = \left\{ \begin{array}{ll}
                0                                       & \mbox{if}\ \  X=\xi \in T_{\al}\\
                \al(4-h \beta + \beta^{2})\p \x_{\ell} & \mbox{if}\ \  X=\x_{\ell} \in T_{\beta} \\
                4(\al + h-\beta)\x_{\ell} & \mbox{if}\ \ X=\p \x_{\ell} \in T_{\gamma}\\
                (h-\beta)(- \al \lambda -2)\p X  & \mbox{if}\ \  X \in T_{\lambda}\\
                (h-\beta)(- \al \mu -2)\p X      & \mbox{if}\ \  X \in T_{\mu}.\\
\end{array}\right.
\end{equation}

From \eqref{eq: 1.1} and \cite[Proposition~$\rm 2$]{BS1}, we obtain the following
\begin{equation}\label{eq: 1.31}
SX = \left\{ \begin{array}{ll}
                (4m+4+h\al-\al^2)\x  & \mbox{if}\ \  X=\x \in T_{\al}\\
                (4m+4+h\be-\be^2)\x_{\ell}  & \mbox{if}\ \  X=\x_{\ell} \in T_{\beta} \\
                (4m+8)\p\x_{\ell}  & \mbox{if}\ \ X=\p \x_{\ell} \in T_{\gamma}\\
                (4m+7+h\la-\la^2) X  & \mbox{if}\ \  X \in T_{\lambda}\\
                (4m+7+h\mu-\mu^2) X      & \mbox{if}\ \  X \in T_{\mu}.\\
\end{array}\right.
\end{equation}

For the case $X=\x$ in \eqref{C-2}, we have $0=\omega(\x)(-8n+8)\x$ which means $\omega(\x)=0$.


For $X \in T_{\gamma}$ and $X \in T_{\mu}$, we have  $h=\be-\al$ and $h=\be$ must be hold. However, this derives $\al=0$ which gives a contradiction.

\begin{remark}\label{remark 1.7}
\rm The Ricci tensor $S$ of real hypersurfaces of Type~$(B)$ in $\GBt$ does not satisfy the {\it recurrent} condition~\eqref{C-2}.
\end{remark}
\vskip 5pt

Hence summing up these considerations, we give a complete proof of our Theorem~$1$ in the introduction.


\section{The recurrent Ricci tensor}\label{section 2}
\setcounter{equation}{0}
\renewcommand{\theequation}{2.\arabic{equation}}
\vspace{0.13in}

Let us assume that the Ricci tensor of a Hopf hypersurface $M$ in $\GBt$ is recurrent. It is given by
\begin{equation}\label{eq: 2.1}
(\N_{X}S)Y=\omega({X})SY
\end{equation}
In this section, we prove Cororally~$2$, given in the introduction. By virtue of lemma~\ref{lemma 1.1}, we know that if $M$ has recurrent Ricci tensor, then the Reeb vector field $\x$ belongs to either $\Q$ or $\QP$.

\vskip 3pt

Next let us consider the case, $\x \in \QP$. Accordingly, we may put $\x=\x_1$.
\vskip 3pt
\begin{lemma}\label{lemma 2.1}
Let $M$ be a Hopf hypersurface in $\GBt$, $m \geq 3$ with vanishing Reeb curvature, that is, $\al=0$. If the Reeb vector field $\x$ belongs to $\QP$
and $M$ has recurrent Ricci tensor, then the shape operator $A$ and the structure tensor field $\p$ commutes with each other i.e., $A\p=\p A$.
\end{lemma}
\vskip 3pt
\begin{proof}
\vskip 3pt
Putting $Y=\x$ into equation~\eqref{eq: 1.3} and
using \eqref{eq: 1.6}, we have
\begin{equation}\label{eq: 2.4}
-6\p AX+h A\p AX+A^2\p AX=4m\omega(X)\x.
\end{equation}

Taking the inner product of \eqref{eq: 2.4} with $\xi$, we have $\omega(X)=0$.

Thus, \eqref{eq: 2.4} becomes

\begin{equation}\label{eq: 2.5}
-6\p AX+h A\p AX+A^2\p AX=0.
\end{equation}

Given that $\x=\xo$, \eqref{eq: 1.14} becomes

\begin{equation}\label{eq: 2.6}
\begin{split}
A\p AX=\p X+\po X-2\etw(X)\xth+2\eh(X)\xtw.
\end{split}
\end{equation}

Applying $A$ to \eqref{eq: 2.6}, and using \eqref{eq: 1.17}, we have

\begin{equation}\label{eq: 2.7}
A^2\p AX=2A\p X.
\end{equation}

Thus, we have
\begin{equation}\label{eq: 2.8}
-6\p AX+h A\p AX+2A\p X=0.
\end{equation}

Taking the symmetric part of \eqref{eq: 2.8}, we have
\begin{equation}\label{eq: 2.9}
6A\p X-h A\p AX-2\p AX=0.
\end{equation}

Combining \eqref{eq: 2.8} and \eqref{eq: 2.9}, we have
$A\p=\p A$.

\end{proof}

Summing up lemmas~\ref{lemma 1.2},~\ref{lemma 1.3},~\ref{lemma 2.1}, \cite[Theorem]{BS2} and \cite[Theorem~$2$]{BS1}, we know that any connected Hopf
hypersurface in $\GBt$ with recurrent Ricci tensor is locally congruent to a real hypersurface $M_{A}$ if the Reeb vector field $\x$ belongs to the distribution
$\QP$. Now we check the converse problem: whether a Hopf hypersurface $M_{A}$ satisfies the given condition \eqref{eq: 2.1} or not. So we assmue that
 $M_{A}$ satisfies \eqref{eq: 2.1}.

Putting $Y=\x$ into \eqref{eq: 1.3}, we obtain
\begin{equation}\label{eq: 2.11}
\begin{split}
&-6\p AX+(h-\al)\al \p AX+hA\p AX+A^2\p AX=\omega(X)(4m+h\al-\al^2)\x.
\end{split}
\end{equation}
\vskip 5pt
Taking $X\in T_{\la}$, we have
\begin{equation}\label{eq: 2.12}
\begin{split}
&\la\big\{-6+(h-\al)\al+h\la+\la^2\big\}\p X=\omega(X)(4m+h\al-\al^2)\x,
\end{split}
\end{equation}
where we have used $\p T_{\la}\subset T_{\la}$ in Type~$A$.
\vskip 5pt
Thus $\la\big\{-6+(h-\al)\al+h\la+\la^2\big\}\p X$ and $\omega(X)(4m+h\al-\al^2)\x$ should vanish respectively.
Using $\la\neq0$ from \cite[Proposition~$3$]{BS1}, as $\p X$ cannot be vanishing, we have

\begin{equation}\label{eq: 2.13}
-6+\al(h-\al)+h\la+\la^2=0.
\end{equation}

Taking $X\in T_{\be}$, \eqref{eq: 2.11} becomes
\begin{equation}\label{eq: 2.14}
\begin{split}
&\be\big\{-6+(h-\al)\al\be+h\be+\be^2\big\}\p X=\omega(X)(4m+h\al-\al^2)\x
\end{split}
\end{equation}
where we have used $\p T_{\be}\subset T_{\be}$ in Type~$A$.
\vskip 5pt
Thus $\be\big\{-6+(h-\al)\al+h\be+\be^2\big\}\p X$ and $\omega(X)(4m+h\al-\al^2)\x$ should be vanishing respectively.
Using $\be\neq0$ from \cite[Proposition~$3$]{BS1}, as $\p X$ cannot be vanishing, we also have

\begin{equation}\label{eq: 2.15}
-6+\al(h-\al)+h\be+\be^2=0.
\end{equation}

Using $\be-\la\neq0$ and combining \eqref{eq: 2.13} and \eqref{eq: 2.15}, we have

\begin{equation}\label{eq: 2.16}
h+\la+\be=0.
\end{equation}

Combining \eqref{eq: 2.13} and \eqref{eq: 2.16}, and applying
\begin{equation*}
\alpha = \sqrt{8}\cot(\sqrt{8}r),\quad  \beta =
\sqrt{2}\cot(\sqrt{2}r),\quad \lambda =-\sqrt{2}\tan(\sqrt{2}r),\quad \mu
= 0
\end{equation*}
with some $r \in (0,\pi/\sqrt{8})$ (see \cite[Proposition~$2$]{BS1}).

We have
\begin{equation}\label{eq: 2.17}
0=-6+\al(h-\al)+h\la+\la^2=4+2\big\{\tan(\sqrt{2}r)-\cot(\sqrt{2}r)\big\}^2+\al^2>0.
\end{equation}
This gives a contradiction.

\begin{remark}\label{remark 2.1}
\rm The Ricci tensor $S$ of real hypersurfaces $M_{A}$ in $\GBt$ does not satisfy the {\it recurrent} condition.
\end{remark}

\vskip 5pt
For $\x\in\Q$, by \cite[Main Theorem]{LS}, we know $g(A\Q,\QP)=0$.
By virtue of Remark~\ref{remark 1.7}, Hopf hypersurface $M_{B}$ does not satisfy the given condition.
%
%

\section{The GTW Reeb recurrent Ricci tensor }\label{section 3}
\setcounter{equation}{0}
\renewcommand{\theequation}{3.\arabic{equation}}
\vspace{0.13in}

In this section, we prove our Theorem~$2$, given in the introduction. Related to Levi-Civita connection $\N$, the generalized Tanaka-Webster connection (from now on, GTW connection) for contact metric manifolds was introduced by Tanno \cite{Tan} as a generalization of the connection defined by Tanaka in \cite{Ta} and, independently, by Webster in \cite{W}. The Tanaka-Webster connection is defined as a canonical affine connection on a non-degenerate, pseudo-Hermitian CR-manifold. A real hypersurface $M$ in a K\"{a}hler manifold has an (integrable) CR-structure associated with the almost contact structure $(\phi ,\xi ,\eta ,g)$ induced on $M$ by the K\"{a}hler structure; however, in general, this CR-structure is not guaranteed to be pseudo-Hermitian. Cho defined GTW connection for a real hypersurface of a K\"{a}hler manifold by
\begin{equation*}
\gtw_{X}Y = \nabla_X Y+F_{X}^{(k)}Y,
\end{equation*}
where constant $k\in {\mathbb R}\setminus\{0\}$ and $F_{X}^{(k)}Y = g(\phi AX,Y)\xi -\eta(Y)\phi AX-k\eta(X)\phi Y$.
$F_{X}^{(k)}$ is a skew-symmetric (1,1) type tensor, that is, $g(F_{X}^{(k)}Y, Z)=-g(Y, F_{X}^{(k)}Z)$ for any tangent vector fields $X, Y$, and~$Z$ on $M$, and is said to be {\it Tanaka-Webster} (or {\it $k$-th-Cho}) {\it operator} with respect to $X$.
In particular, if the real hypersurface satisfies $A\phi +\phi A=2k\phi$, then the GTW connection $\hat {\nabla}^{(k)}$ coincides with the Tanaka-Webster connection (see \cite{CH1}, \cite{CH2}, \cite{JKLS}).
\vskip 5pt
The Ricci tensor $S$ is said to be {\it generalized Tanaka-Webster parallel} (in short, {\it GTW parallel}) if the covariant derivative in GTW connection $\gtw$ of $S$ along any $X$ vanishes, that is, if $(\gtw_{X}S)Y=0$.
\vskip 5pt
From the definition of $\gtw$ and $(\gtw_{X}S)Y$, we have
\begin{equation}\label{eq: 3.2}
\begin{split}
(\gtw_{X}S)Y&=(\N_{X}S)Y+F_{X}^{(k)}(SY)-SF_{X}^{(k)}Y\\
            &=\omega(X)SY.
\end{split}
\end{equation}


The condition \eqref{eq: 3.2} is specified as follow:

\begin{equation}\label{eq: 3.4}
\begin{split}
(\gtw_{X}S)Y & = (\N_{X}S)Y\\
                 &\quad + g(\phi AX,SY)\xi -\eta(SY)\phi AX-k\eta(X)\phi SY\\
                 &\quad - g(\phi AX,Y)S\xi +\eta(Y)S\phi AX+k\eta(X)S\phi Y\\
                 &=\omega(X)SY.
\end{split}
\end{equation}

The Ricci tensor $S$ is said to be {\it GTW Reeb parallel} if the covariant derivative in GTW connection $\gtw$ of $S$ along the Reeb direction vanishes, that is, if $(\gtw_{\x}S)Y=0$.
Furthermore, GTW Reeb recurrent Ricci tensor is given by
\begin{equation}\label{C-4}
\gtw_{\x}S=\omega(\x)S.
\tag{C-4}
\end{equation}

\begin{lemma}\label{lemma 3.1}
Let $M$ be a Hopf hypersurface in $\GBt$, $m \geq 3$. If $M$ has GTW Reeb recurrent Ricci tensor, then the Reeb vector field $\x$ belongs to either the distribution $\Q$ or the distribution $\QP$.
\end{lemma}

\begin{proof}
We write
\begin{equation}\label{xi}
\x = \eta(X_{0})X_{0}+\eta(\x_{1})\x_{1}
\tag{*}
\end{equation}
for some unit vectors $X_{0} \in \Q$ and $\x_{1} \in \QP$.

Putting $Y=\x$ into \eqref{C-4} and applying $\p$ to \eqref{C-4}, we have

\begin{equation}\label{eq: 3.8}
\begin{split}
-4(\al+k)\eo(\x)\big\{\xo+\e(\xo)\x\big\}=-4\omega(\x)\eo(\x)\p\xo.
\end{split}
\end{equation}

Taking an inner product with $X_{0}$, we have

\begin{equation}\label{eq: 3.9}
-4(\al+k)\eo^2(\x)\e(X_{0})=0.
\end{equation}
From this, we have the following three cases.

\vskip 3pt
\noindent {\bf Case 1\,:} \quad $\al = -k$.
\newline By the equation $Y{\alpha}=({\x}{\alpha}){\eta}(Y)-4{\SN}{\ENK}{\EN}({\p}Y)$
in \cite[Lemma~$1$]{BS1}, we obtain easily that $\x$ belongs to either $\Q$ or $\QP$ (see \cite{PS}).

\vskip 3pt
\noindent {\bf Case 2\,:} \quad $\eta(\x_{1})=0$.
\vskip 3pt

By the notation \eqref{xi} related to the Reeb vector field, we see that $\x$ belongs to the distribution $\Q$.

\vskip 3pt
\noindent {\bf Case 3\,:} \quad $\E(X_0)=0$.
\vskip 3pt

This case implies that $\x$ belongs to the distribution $\QP$.
\vskip 5pt
\noindent Accordingly, summing up these cases, it completes the proof of our Lemma.
\end{proof}
\vskip 3pt


As we know,
\begin{equation}\label{eq: 3.10}
\begin{split}
(\gtw_{\x}S)Y=(\N_{\x}S)Y+k(S\p-\p S)Y
\end{split}
\end{equation}

Next let us consider the case, $\x \in \QP$. Accordingly, we may put $\x=\x_1$.

\begin{lemma}\label{lemma 3.2}
Let $M$ be a Hopf hypersurface in $\GBt$, $m \geq 3$.
When the Reeb vector field $\x$ belongs to the distribution $\QP$, if $M$ has the GTW Reeb recurrent Ricci tensor, that is, $(\gtw_{\x}S)X=\omega(\x)SX$ $(\al \neq 2k)$, then $S\p=\p S$.
\end{lemma}
\vskip 3pt

\begin{proof}
Using \eqref{eq: 1.16} and \eqref{eq: 1.18}, then \eqref{eq: 3.2} becomes

\begin{equation}\label{eq: 3.11}
\begin{split}
&(\x h)AX +(\frac{\al}{2}-k)(\p SX-S\p X) +(h-2\al)(\x\al)\e(X)\x=\omega(\x)SX.
\end{split}
\end{equation}

Substituting $X$ to $A X$ into \eqref{eq: 3.11} and
applying $A$ to \eqref{eq: 3.11} and combining them, we have $(\p S-S\p)A=A(\p S-S\p)$.
By \cite[Lemma 1.5]{PaSW}, we have $S\p=\p S$.

\end{proof}

Summing up these discussions, we conclude that if a Hopf hypersurface $M$ in complex two-plane Grassmannians $\GBt$, $m \geq 3$, satisfying $(\gtw_{X}S)Y=\omega(X)SY$ then $M$ is of Type~$(A)$.
%
Hereafter, let us check whether $S$ of a model space of $M_{A}$ satisfies the Reeb parallelism with respect to $\gtw$ by
\cite[Proposition~$\rm 3$]{BS1} (see \cite{LCW}).
From these two equations, it follows that
\begin{equation}\label{eq: 3.13}
(\gtw_{\x}S)X  = \left\{ \begin{array}{ll}
                0                    & \mbox{if}\ \  X=\xi \in T_{\al}\\
                0                    & \mbox{if}\ \  X=\x_{\ell} \in T_{\beta}=\text{Span}\{\x_{\ell}|\, \ell=1,2,3\} \\
                0                    & \mbox{if}\ \  X \in T_{\lambda}\\
                0                    & \mbox{if}\ \  X \in T_{\mu}.\\
\end{array}\right.
\end{equation}

\vskip 5pt

Consider \eqref{eq: 1.28} and $X=\xi \in T_{\al}$; thus, $S\x=(4m+h\al-\al^2)\x$. Thus, $\omega(\xi)=0$.

\noindent Summing up all cases mentioned above, we can assert that if $\omega(\xi)=0$, then $S$ of $M_{A}$ in $\GBt$ is GTW Reeb parallel.

\begin{remark}\label{remark 3.1}
\rm The Ricci tensor $S$ of real hypersurfaces $M_{A}$ in $\GBt$ satisfies the {\it GTW Reeb parallel} condition if $\omega(\xi)=0$.
\end{remark}
\vskip 3pt

For $\x\in\Q$, by \cite[Main Theorem]{LS}, we know $g(A\Q,\QP)=0$.
\vskip 3pt
Now let us consider our problem for a model space $M_{B}$. In order to do this, let us calculate the fundamental equation related to the covariant derivative of $S$ of $M_{B}$ along the direction of $\x$ in GTW connection. On $T_{x}M_{B}$, $x \in M_{B}$, since $\x \in \Q$ and $h=\text{Tr}(A)=\al+(4n-1)\beta$ is a constant, \eqref{C-4} is reduced to
\begin{equation*}\label{eq: 3.14}
\begin{split}
(\gtw_{\x}S)X & =  4(k-\al) \sum_{\nu =1}^3 \Big \{\eta_{\nu}(\p X)\x_{\nu} -  \EN(X)\pn\x \Big\} \\
                & \ \ - h(\nabla_{\x}A)X + (\nabla_{\x}A)AX + A(\nabla_{\x}A)X \\
                & \ \ + kh\Ph AX - k\Ph A^2 X - khA\Ph X + kA^2 \Ph X.
\end{split}
\end{equation*}
Moreover, by the equation of Codazzi \cite{BS1} and \cite[Proposition~$\rm 2$]{BS1} we obtain that for any $X \in T_{x}M_{B}$
\begin{equation}\label{eq: 3.15}
\begin{split}
(\N_{\x}A)X & =\al \p AX - A \p AX + \p X - \sum_{\nu=1}^{3}\big\{\e_{\nu}(X)\p_{\nu}\x + 3g(\p_{\nu}\x, X)\x_{\nu}\big\}\\
            & = \left\{ \begin{array}{ll}
                0                       & \mbox{if}\ \  X \in T_{\al}\\
                \al \beta \p \x_{\ell} & \mbox{if}\ \  X \in T_{\beta}=\text{Span}\{\x_{\ell}|\, \ell=1,2,3\}\\
                -4 \x_{\ell}           & \mbox{if}\ \  X \in T_{\gamma}=\text{Span}\{\p \x_{\ell}|\, \ell=1,2,3\}\\
                (\al \lambda + 2) \p X  & \mbox{if}\ \  X \in T_{\lambda}\\
                (\al \mu + 2) \p X      & \mbox{if}\ \  X \in T_{\mu}.\\
\end{array}\right.
\end{split}
\end{equation}
From these two equations, it follows that
\begin{equation}\label{eq: 3.16}
(\gtw_{\x}S)X  = \left\{ \begin{array}{ll}
                0                                       & \mbox{if}\ \  X=\xi \in T_{\al}\\
                (\al-k)(4-h \beta + \beta^{2})\p \x_{\ell} & \mbox{if}\ \  X=\x_{\ell} \in T_{\beta} \\
                \big\{4(\al -k) + (h-\beta)(4+k\beta)\big\}\x_{\ell} & \mbox{if}\ \ X=\p \x_{\ell} \in T_{\gamma}\\
                (h-\beta)(k \lambda - k \mu - \al \lambda -2)\p X  & \mbox{if}\ \  X \in T_{\lambda}\\
                (h-\beta)(k \mu - k \lambda - \al \mu -2)\p X      & \mbox{if}\ \  X \in T_{\mu}.\\
\end{array}\right.
\end{equation}
Therefore, we see that $M_{B}$ has Reeb parallel GTW-Ricci tensor, when $\al$ and $h$ satisfies the conditions $\al = k$ and $h-\beta=0$, which means $r=\frac{1}{2}\cot^{-1}(\frac{-k}{4(2n-1)})$. Moreover, this radius $r$ satisfies our condition $\al \neq 2k$.
%
Secondly, we check whether a model space $M_{B}$ satisfies the condition of GTW Reeb recurrent Ricci tensor. In this case, \eqref{eq: 3.10} becomes

\begin{equation}\label{eq: 3.17}
\begin{split}
&-3\p AX-\sum_{\nu =1}^3g(AX,\XN)\PN \x+(h-\al)\al\p AX+hA\p AX+A^2\p AX\\
&=\omega(X)(4m+h\al-\al^2)\x.
\end{split}
\end{equation}

Taking the inner product with $\x$, we get $(4m+h\al-\al^2)\omega(X)=0$ which means

\begin{equation}\label{eq: 3.18}
\begin{split}
-3\p AX-\sum_{\nu =1}^3g(AX,\XN)\PN \x+(h-\al)\al\p AX+hA\p AX+A^2\p AX=0.
\end{split}
\end{equation}

\begin{remark}\label{remark 3.2}
\rm The Ricci tensor $S$ of any real hypersurface $M_{B}$ in $\GBt$ satisfies the {\it GTW Reeb parallel} condition.
\end{remark}

\vskip 3pt
Consider $X=\xo\in T_\beta$, we get
\begin{equation}\label{eq: 3.19}
\begin{split}
(-2+h\al-\al^2-h\be)\p \xo=0.
\end{split}
\end{equation}
\vskip 3pt
The coefficient of left term is less than 0, i.e., $-2+h\al-\al^2-h\be=-2-4(4n-2)-(4n-1)\be^2<0$. This means $\p \xo=0$ which makes a contradiction.

\begin{remark}\label{remark 3.3}
\rm The Ricci tensor $S$ of a real hypersurface $M_{B}$ in $\GBt$ does not satisfy the {\it Proper GTW Reeb recurrent} condition.
\end{remark}

\section{GTW recurrent Ricci tensor }\label{section 4}
\setcounter{equation}{0}
\renewcommand{\theequation}{4.\arabic{equation}}

\vspace{0.13in}
By virtue of \ref{lemma 3.1}, if $M$ has the GTW recurrent Ricci tensor~\eqref{eq: 3.4}
($\al\neq 2k$), then the Reeb vector field $\x$ belongs to either $\Q$ or $\QP$.
In addition, by virtue of lemma~\ref{lemma 3.2}, if $\x$ belongs to $\QP$, we have $S\p=\p S$.
Now we check the converse problem whether a real hypersurface $M_{A}$ satisfies the given condition \eqref{eq: 3.4} or not.

Putting $Y=\x$ into \eqref{eq: 3.4}, we get

\begin{equation}\label{eq: 4.1}
\begin{split}
(\N_{X}S)\x+F_{X}^{(k)}(S\x)-SF_{X}^{(k)}\x=\omega(X)S\x.
\end{split}
\end{equation}

Taking the inner product of \eqref{eq: 4.1} with $\x$, consider $(\N_{\x}S)\x=0$, $F_{X}^{(k)}$ is skew symmetric and $S\x=(4m+h\al-\al^2)\x$, we have
$(4m+h\al-\al^2)\omega(X)=0$, where $h=\al+2\be+(2m-2)(\lambda+\mu)$.

\begin{equation}
\begin{split}
4m+h\al-\al^2&=4m+2\al\be+(2m-2)\al\la\\
               &=4\{\cot^2(\theta)+(m-1)\tan^2(\theta)\}\\
               & \geq 8\sqrt{(m-1)}\\
               &>0.
\end{split}
\end{equation}

This gives
\begin{equation}\label{eq: 4.2}
\omega(X)=0.
\end{equation}
Putting $Y=\xi$ into \eqref{eq: 4.1}, we have

\begin{equation}\label{eq: 4.3}
-6\p AX+(h-\al)\al \p AX+hA\p AX+A^2\p AX-\sigma\p AX+S\p AX=0,
\end{equation}
where $\sigma = 4m+h\al-\al^2$.

Putting $X\in T_{\la}$ into \eqref{eq: 4.1}, we have $2h\la=0$ which means

\begin{equation}\label{eq: 4.4}
h=0.
\end{equation}

Consider $Y=\xth\in T_{\mu}$ into \eqref{eq: 4.2}, by \eqref{eq: 4.3} and \eqref{eq: 4.4}, we have

\begin{equation}\label{eq: 4.5}
\begin{split}
&-4\al \e(X)\x-3\po AX+\pth\p AX-\be(\N_X A)\xth\\
             &-\beta A(\N_X A)\xth+(6-\be^2+\al^2)\eh(AX)\x\\
             &=0.
\end{split}
\end{equation}

Taking the inner product with $\xtw$ of \eqref{eq: 4.5}, we have $3\be\eh(X)=0$. This means $3\be\xth=0$, and gives a contradiction.
Putting $Y\in T_{\mu}$ into \eqref{eq: 4.1}, we have

\begin{remark}\label{remark 4.1}
\rm The Ricci tensor $S$ of a real hypersurface $M_{A}$ in $\GBt$ does not satisfy the {\it GTW recurrent} condition.
\end{remark}

Now we check the converse problem, that is, a real hypersurface $M_{B}$ satisfies the given condition \eqref{eq: 3.4} or not.
Hereafter, let us check whether $M_{B}$ satisfies the condition of GTW recurrent Ricci tensor.

\eqref{eq: 1.3} becomes
\begin{equation}\label{eq: 4.6}
\begin{split}
&-3\p AX-\sum_{\nu =1}^3g(AX,\XN)\PN \x+(h-\al)\al\p AX+hA\p AX+A^2\p AX\\
&=\omega(X)(4m+h\al-\al^2)\x.
\end{split}
\end{equation}

Taking the inner product of \eqref{eq: 4.6} with $\x$, we get $\omega(X)=0$, which means

\begin{equation*}\label{eq: 4.7}
-3\p AX-\sum_{\nu =1}^3g(AX,\XN)\PN \x+(h-\al)\al\p AX+hA\p AX+A^2\p AX=0.
\end{equation*}
\vskip 3pt
Consider $X=\xo\in T_\beta$ into above equation, we get
\begin{equation}\label{eq: 4.8}
(-2+h\al-\al^2-h\be)\p \xo=0.
\end{equation}
\vskip 3pt

Since $-2+h\al-\al^2-h\be=-2-4(4n-2)-(4n-1)\be^2<0$, \eqref{eq: 4.8} means $\p \xo=0$. This is a contradiction.
\vskip 3pt
\begin{remark}\label{remark 4.2}
\rm The Ricci tensor $S$ of real hypersurfaces $M_{B}$ in $\GBt$ does not satisfy the {\it GTW recurrent} condition.
\end{remark}
Summing up these assertions, we give a complete proof of Cororally~$2$ in the introduction.


\end{document}